\documentclass[a4paper,12pt,reqno]{amsart}

\usepackage{amsmath,amssymb,amsthm}
\usepackage{comment}

\numberwithin{equation}{section}
\allowdisplaybreaks
\newtheorem{thm}{Theorem}[section]
\newtheorem{lem}[thm]{Lemma}
\begin{document}
	
	\title{On the Third Hankel Determinant for Inverse Coefficients of Starlike Functions: A Bernstein Polynomial Approach}
	\author{Vasudevarao Allu}
	\address{Vasudevarao Allu, Department of Mathematics, School of Basic Sciences,
		Indian Institute of Technology Bhubaneswar, Jatni Road, Argul,
		Bhubaneswar 752050, Odisha, India}
	\email{avrao@iitbbs.ac.in}
	
	\author{Shobhit Kumar}
	\address{Shobhit Kumar, Department of Mathematics, School of Basic Sciences,
		Indian Institute of Technology Bhubaneswar, Jatni Road, Argul,
		Bhubaneswar 752050, Odisha, India}
	\email{a21ma09007@iitbbs.ac.in}
	\subjclass[2020]{30C45, 30C50, 41A10}
	\keywords{Starlike functions, inverse coefficients, Hankel determinant, Carath\'eodory class, Bernstein polynomial}
	\date{}
	\begin{abstract}
		Let $\mathcal{A}$ denote the class of normalized analytic functions $f$ in the open unit disk defined as
		$
		\mathbb{D}:=\{z\in\mathbb{C}:|z|<1\}
		$
		with $f(0)=0$ and $f'(0)=1$. A function $f\in\mathcal{A}$ is said to be starlike if $f(\mathbb{D})$ is starlike domain. By using  the Bernstein polynomial method to obtain the required maximum estimate, we establish  sharp upper bound for the third Hankel determinant corresponding to the inverse coefficients of starlike univalent ({\it i.e.}, one-to-one) functions in the unit disk $\mathbb{D}$. 
	\end{abstract}
	\maketitle
	\markboth{V.~Allu and S.~Kumar}{Inverse Hankel determinant for starlike functions}
	
\section{Introduction}

Let $\mathcal{A}$ denote the class of normalized analytic functions $f$ in the open unit disk
$\mathbb{D}:=\{z\in\mathbb{C}:|z|<1\}$ having the Taylor series  expansion
\begin{equation}
	\label{eq:1.1}
	f(z)=z+\sum_{n=2}^{\infty}a_n z^n,\quad \text{ for } z\in\mathbb{D},
\end{equation}
and let $\mathcal{S}$ be the subclass of $\mathcal{A}$ consisting of functions that are univalent ({\it i.e.}, one-to-one) in the unit disk $\mathbb{D}$. Two important subclasses of $\mathcal{S}$ are the classes of starlike and convex functions, defined respectively by
\[
\mathcal{S}^{*}
=
\left\{
f \in \mathcal{A} : \operatorname{Re}\left( \frac{z f'(z)}{f(z)} \right) > 0,\ z \in \mathbb{D}
\right\},
\]
and
\[
\mathcal{C}
=
\left\{
f \in \mathcal{A} : \operatorname{Re}\left( 1 + \frac{z f''(z)}{f'(z)} \right) > 0,\ z \in \mathbb{D}
\right\}.
\]
For every function \(f\in\mathcal{S}\), the local inverse \(f^{-1}\) exists in a neighborhood of the origin and admits the expansion
\[
f^{-1}(w)=w+A_2w^2+A_3w^3+\cdots .
\]
The notion of Hankel determinants for functions in \(\mathcal{A}\) was introduced by Pommerenke \cite{Pommerenke1966}. For \(n,q\in\mathbb{N}\), the \(q\)th Hankel determinant is given by
\[
H(q,n)(f)
:=
\begin{vmatrix}
	a_n     & a_{n+1}   & \cdots & a_{n+q-1} \\
	a_{n+1} & a_{n+2}   & \cdots & a_{n+q}   \\
	\vdots  & \vdots    & \ddots & \vdots    \\
	a_{n+q-1} & a_{n+q} & \cdots & a_{n+2q-2}
\end{vmatrix}.
\]
Among the simplest cases, the determinants
\[
H(2,1)(f)=a_2^2-a_3
\qquad\text{and}\qquad
H(2,2)(f)=a_2a_4-a_3^2
\]
have received considerable attention in the literature; see, for example,
(see \cite{ChoKowalczykKwonLeckoSim2018,LeeRavichandranSupramanian2013}).\\
\par The study of the third Hankel determinant is substantially more involved. Over the past several years, a number of authors have obtained important results with regards to 
\[
H(3,1)(f)=2a_2a_3a_4-a_3^3-a_4^2+a_3a_5-a_2^2a_5
\]
for various subclasses of analytic and univalent functions (see
\cite{BangaKumar2020,KowalczykLeckoLeckoSim2018,KwonLeckoSim2019,LeckoSimSmiarowska2019,RiazRazaThomas2022}).
In particular, sharp bounds for this determinant have been established for both the class of  convex and starlike functions in recent works such as
(see \cite{KowalczykLeckoSim2018convex,KowalczykLeckoThomas2022}).\\
\par In the present paper, we investigate the third Hankel determinant formed from the inverse coefficients of functions belonging to \(\mathcal{S}^{*}\). Our main theorem gives a sharp estimate for \(\bigl|H(3,1)(f^{-1})\bigr|\), with equality attained by the Koebe function. The proof relies on the Carath\'eodory representation of starlike functions, together with the coefficient parametrization due to Libera and Z{\l}otkiewicz \cite{LiberaZlotkiewicz1982} and Kwon, \textit{et al.} \cite{KwonLeckoSim2018Caratheodory}.\\
\par A crucial step in the proof is the estimation of a polynomial on a rectangular box. For this purpose, we employ the Bernstein polynomial method,  for establishing the positivity or nonnegativity of polynomials on compact intervals and rectangular boxes; see Cargo and Shisha \cite{CargoShisha1966} and Garloff \cite{Garloff1986}. This approach provides an effective tool for reducing the required maximization problem to the inspection of Bernstein coefficients.

Let \(\mathcal{P}\) denote the class of analytic functions \(p\) in \(\mathbb{D}\) of the form
\begin{equation}
	\label{eq:1.2}
	p(z)=1+\sum_{n=1}^{\infty}c_n z^n
\end{equation}
with positive real part in \(\mathbb{D}\). The following lemma will play a central role in the proof of our main result, since it provides the standard representations of the coefficients \(c_2\), \(c_3\), and \(c_4\).
	
\begin{lem}{\cite{LiberaZlotkiewicz1982}}
	\label{lem:1.1}
		Let $p \in \mathcal{P}$ and be given by \eqref{eq:1.2} with $c_1 \ge 0$. Then
		\begin{align*}
			2 c_2 &= c_1^2 + \gamma \bigl(4 - c_1^2\bigr), \\[2mm]
			4 c_3 &= c_1^3 + 2\bigl(4 - c_1^2\bigr)c_1 \gamma
			- \bigl(4 - c_1^2\bigr)c_1 \gamma^2
			+ 2\bigl(4 - c_1^2\bigr)\bigl(1 - |\gamma|^2\bigr)\eta, \\[2mm]
			8 c_4 &= c_1^4
			+ \bigl(4 - c_1^2\bigr)\gamma\bigl(c_1^2(\gamma^2 - 3\gamma + 3) + 4\gamma\bigr)\\[1mm]
			&\quad- 4\bigl(4 - c_1^2\bigr)\bigl(1 - |\gamma|^2\bigr)
			\left(c_1(\gamma - 1)\eta + \overline{\gamma}\,\eta^2
			- (1 - |\eta|^2)\rho\right)
		\end{align*}
		for some $\gamma,\eta,\rho$ such that $|\gamma| \le 1$, $|\eta| \le 1$ and $|\rho| \le 1$.
	\end{lem}
	
	\section{Main Result}
	
	\begin{thm}
		\label{thm:2.1}
		Let $f \in \mathcal{S}^{*}$  be given by \eqref{eq:1.1}. Then
		\[
		\bigl| H(3,1)(f^{-1}) \bigr| \le 1.
		\]
		The inequality is sharp for the Koebe function.
	\end{thm}
	
	\begin{proof}
		Let $f \in \mathcal{S}^{*}$. Then there exists a function $p \in \mathcal{P}$ such that
		\begin{equation}\label{eq:2.1}
			z f'(z) = f(z)\,p(z).
		\end{equation}
		Note that since $f\bigl(f^{-1}(w)\bigr) = w$, it follows that
		\begin{equation}\label{eq:2.2}
			\left.
			\begin{aligned}
				A_2 &= -a_2,\\
				A_3 &= 2a_2^2 - a_3,\\
				A_4 &= 5 a_2 a_3 - 5 a_2^3 - a_4,\\
				A_5 &= 14 a_2^4 - 21 a_3 a_2^2 + 6 a_2 a_4 + 3 a_3^2 - a_5.
			\end{aligned}
			\right\}
		\end{equation}
		Comparing coefficients of \eqref{eq:2.1}, we obtain
		\begin{equation}\label{eq:2.3}
			\left.
			\begin{aligned}
				a_2 &= c_1,\\
				a_3 &= \frac{1}{2} c_2 + \frac{1}{2} c_1^2,\\
				a_4 &= \frac{1}{6} c_1^3 + \frac{1}{2} c_1 c_2 + \frac{1}{3} c_3,\\
				a_5 &= \frac{1}{24} c_1^4 + \frac{1}{4} c_2 c_1^2 + \frac{1}{3} c_1 c_3
				+ \frac{1}{8} c_2^2 + \frac{1}{4} c_4.
			\end{aligned}
			\right\}
		\end{equation}
		Using \eqref{eq:2.2} and \eqref{eq:2.3}, we obtain that
		\begin{align}\label{eq:2.4}
			144\, H(3,1)\,(f^{-1}) &=
			17 c_1^6 - 51 c_1^4 c_2 + 45 c_1^2 c_2^2 - 27 c_2^3
			+ 8 c_1^3 c_3 + 24 c_1 c_2 c_3\notag\\
			&\quad - 16 c_3^2
			- 18 c_1^2 c_4 + 18 c_2 c_4.
		\end{align}
		Since,  the class $\mathcal{S}^{*}$ is rotationally invariant. For $f \in \mathcal{S}^*$, the rotational functions 
		\[
		f_{\theta}(z):=e^{-i\theta}f(e^{i\theta}z)
		\]
		belongs to the class $\mathcal{S}^*.$ Moreover, if
		\begin{align*}
		f^{-1}(w)=w+A_{2}w^{2}+A_{3}w^{3}+A_{4}w^{4}+A_{5}w^{5}+\cdots,
		\end{align*}
		then
		\begin{align*}
		f_{\theta}^{-1}(w)=e^{-i\theta}f^{-1}(e^{i\theta}w)
		= w+e^{i\theta}A_{2}w^{2}+e^{2i\theta}A_{3}w^{3}+e^{3i\theta}A_{4}w^{4}
		+e^{4i\theta}A_{5}w^{5}+\cdots .
		\end{align*}
		Thus, we have 
		\[
		H(3,1)(f_{\theta}^{-1})=e^{6i\theta}H(3,1)(f^{-1}),
		\]
	 therefore, 
		\[
		|H(3,1)(f_{\theta}^{-1})|=|H(3,1)(f^{-1})|.
		\]
		Therefore, choosing $\theta$ suitably, we may assume that $c_{1}\in[0,2]$. Substituting the values of $c_2,c_3$ and $c_4$
		from Lemma~\ref{lem:1.1} into \eqref{eq:2.4}, and collecting
		the coefficients of $\eta$, $\eta^2$ and $\rho$, we may rewrite \eqref{eq:2.4} in the form
		\begin{equation}\label{eq:2.5}
			1152\, H(3, 1)\,(f^{-1})=\widetilde{A}+ \widetilde{B} \eta + \widetilde{C} \eta^2+ \widetilde{D}\rho
		\end{equation}
		where,
		\begin{align*}
			\widetilde{A}&=18\,c_{1}^{6}
			\;+\;
			51\,\gamma\,c_{1}^{4}\,(-4 + c_{1}^{2})
			\;+\;
			\gamma^{4}\,c_{1}^{2}\,(-4 + c_{1}^{2})^{2}
			\,+\,
			\gamma^{2}c_{1}^{2}\,(688 - 368c_{1}^{2} + 49c_{1}^{4})
			\;\\[3pt]
			&\quad+\gamma^{3}\,(-1152 + 704c_{1}^{2} - 172c_{1}^{4} + 17c_{1}^{6}), \\[3pt]
			\widetilde{B}&=4\,(-1 + |\gamma|^{2})\,c_{1}\,(-4 + c_{1}^{2})\,
			\Big( 3c_{1}^{2} + \gamma^{2}(-4 + c_{1}^{2}) + 4\gamma(5 + c_{1}^{2}) \Big),\\[3pt]
			\widetilde{C}&=-4\,(-1 + |\gamma|^{2})\,(-4 + c_{1}^{2})\,
			\Bigl(
			-8(-4 + c_{1}^{2})
			+ 8\,|\gamma|^{2}(-4 + c_{1}^{2})
			- 9\,(c_{1}^{2} + \gamma(-4 + c_{1}^{2}))\,\overline{\gamma}
			\Bigr),\\[2pt]
			\widetilde{D}&=36\,(-1 + |\gamma|^{2})(-1 + |\eta|^{2})\,(-4 + c_{1}^{2})\,
			\bigl(c_{1}^{2} + \gamma(-4 + c_{1}^{2})\bigr).
		\end{align*}
		Now, taking
		$
		x:=|\gamma|\in[0,1],\, y:=|\eta|\in[0,1].
		$
		Then, using the triangle inequality and the fact that $|\rho|\le 1$, we obtain
		\[
		1152\,\bigl|H(3,1)(f^{-1})\bigr|
		\le |\widetilde{A}|+|\widetilde{B}|\,|\eta|+|\widetilde{C}|\,|\eta|^{2}+|\widetilde{D}|\,|\rho|
		\le H(c_{1},x,y).
		\]
		For convenience, take $c_{1}=p_{1}$, where
		\begin{align}\label{eq:2.6}
			H(p_{1}, x, y) :=\;
			&\,18 p_{1}^{6}
			+ 51x\,p_{1}^{4}(4 - p_{1}^{2})
			+ x^{2} p_{1}^{2}(688 - 368 p_{1}^{2} + 49 p_{1}^{4}) \notag\\[4pt]
			& - x^{3}\!\left(-1152 + 704 p_{1}^{2} - 172 p_{1}^{4} + 17 p_{1}^{6}\right)
			+ x^{4} p_{1}^{2}(-4 + p_{1}^{2})^{2} \notag\\[4pt]
			& + 4\, y\,(1 - x^{2}) p_{1}(4 - p_{1}^{2})
			\left(3 p_{1}^{2} + x^{2}(4 - p_{1}^{2}) + 4x(5 + p_{1}^{2})\right) \notag\\[4pt]
			& + 4\,y^{2}\,(1 - x^{2})(4 - p_{1}^{2})
			\left(8(1 - x^{2})(4 - p_{1}^{2})
			+ 9(p_{1}^{2}x + x^{2}(4 - p_{1}^{2}))\right) \notag\\[4pt]
			& + 36(1 - x^{2})(1 - y^{2})(4 - p_{1}^{2})
			\left(p_{1}^{2} + x(4 - p_{1}^{2})\right).
		\end{align}
		We now determine the maximum value of \(H(p_{1},x,y)\) on the cuboid
		$
		\mathcal{D}:=\{(p_1, x, y)\in \mathbb{R}^3: 0\le p_{1}\le 2,\, 0\le x\le 1,\, 0\le y\le 1\}.
		$
		Our aim is to prove 
		\[
		\max_{\mathcal{D}} H(p_{1},x,y)=1152.
		\]
		Now, consider
		\[
		\Delta(p_{1},x,y):=1152-H(p_{1},x,y).
		\]
		A direct calculation shows that
		\[
		\Delta(p_{1},x,y)
		=(2-p_{1})(2+p_{1})(1-x)\,\Phi(p_{1},x,y),
		\]
		where
		\[
		\begin{aligned}
			\Phi(p_{1},x,y)=\;&
			-p_{1}^{4}x^{3}+16p_{1}^{4}x^{2}-33p_{1}^{4}x+18p_{1}^{4} +4p_{1}^{3}x^{3}y-12p_{1}^{3}x^{2}y-28p_{1}^{3}xy\\
			&-12p_{1}^{3}y +4p_{1}^{2}x^{3}y^{2}+4p_{1}^{2}x^{3}-68p_{1}^{2}x^{2}y^{2}-64p_{1}^{2}x^{2} -4p_{1}^{2}xy^{2}\\
			&+72p_{1}^{2}x+68p_{1}^{2}y^{2}+36p_{1}^{2}-16p_{1}x^{3}y-96p_{1}x^{2}y-80p_{1}xy \\
			&-16x^{3}y^{2}+128x^{2}y^{2}+144x^{2}+16xy^{2}+144x-128y^{2}+288 .
		\end{aligned}
		\]
		Let
		$
		p_{1}=2u,$  where $0\le u\le 1.
		$
		Then \(\Phi(2u,x,y)\) is a polynomial of degree at most \(4\) in \(u\), degree at most \(3\) in \(x\), and degree at most \(2\) in \(y\). Hence it can be written in the Bernstein form
		\[
		\Phi(2u,x,y)
		=
		\sum_{i=0}^{4}\sum_{j=0}^{3}\sum_{k=0}^{2}
		b_{ijk}\,B_{i}^{4}(u)B_{j}^{3}(x)B_{k}^{2}(y),
		\]
		where
		\[
		B_{m}^{n}(t)=\binom{n}{m}t^{m}(1-t)^{n-m},\qquad 0\le t\le 1.
		\]
		The corresponding Bernstein coefficient matrices are
		\[
		M_{0}=
		\begin{pmatrix}
			288&336&432&576\\[3mm]
			288&336&432&576\\[3mm]
			312&376&\frac{4264}{9}&608\\[3mm]
			360&456&\frac{1672}{3}&672\\[3mm]
			720&688&704&768
		\end{pmatrix},
		\]
		and the maximum entry of \(M_{0}\) is
		\[
		\max M_{0}=768.
		\]
		
		\[
		M_{1}=
		\begin{pmatrix}
			288&336&432&576\\[3mm]
			288&\frac{988}{3}&\frac{1232}{3}&528\\[3mm]
			312&\frac{1088}{3}&\frac{3880}{9}&512\\[3mm]
			348&\frac{1244}{3}&\frac{1376}{3}&480\\[3mm]
			672&576&480&384
		\end{pmatrix},
		\]
		and the maximum entry of \(M_{1}\) is
		\[
		\max M_{1}=672.
		\]
		
		\[
		M_{2}=
		\begin{pmatrix}
			160&\frac{640}{3}&\frac{1072}{3}&576\\[3mm]
			160&200&\frac{944}{3}&480\\[3mm]
			\frac{688}{3}&\frac{2440}{9}&\frac{3080}{9}&416\\[3mm]
			344&384&\frac{1112}{3}&288\\[3mm]
			768&608&352&0
		\end{pmatrix},
		\]
		and the maximum entry of \(M_{2}\) is
		\[
		\max M_{2}=768.
		\]
		Since every entry of \(M_{0}\), \(M_{1}\), and \(M_{2}\) is non-negative, and since each Bernstein basis polynomial is non-negative on \([0,1]\), it follows that
		\begin{align*}
		\Phi(2u,x,y)\ge 0
		\qquad \text{for all } (u,x,y)\in[0,1]\times[0, 1]\times[0, 1].
		\end{align*}
		Equivalently,
		\begin{align*}
		\Phi(p_{1},x,y)\ge 0
		\qquad \text{for all } (p_{1},x,y)\in[0,2]\times[0,1]\times[0,1].
		\end{align*}
		Therefore,
		\[
		\Delta(p_{1},x,y)
		=(2-p_{1})(2+p_{1})(1-x)\Phi(p_{1},x,y)\ge 0,
		\]
		because
		\[
		2-p_{1}\ge 0,\qquad 2+p_{1}>0,\qquad 1-x\ge 0
		\]
		throughout the domain. Hence
		\[
		H(p_{1},x,y)\le 1152, \qquad \text{ for } \quad 0\le p_{1}\le 2,\quad 0\le x\le 1,\quad 0\le y\le 1.
		\]
		To see that the bound $1152$ is sharp, we note that
		\[
		H(p_{1},1,y)=1152
		\qquad \text{for all } 0\le p_{1}\le 2,\ \ 0\le y\le 1,
		\]
		and also
		\[
		H(2,x,y)=1152
		\qquad \text{for all } 0\le x\le 1,\ \ 0\le y\le 1.
		\]
		Thus,
		\begin{align*}
		\max_{0\le p_{1}\le 2,\;0\le x\le 1,\;0\le y\le 1} H(p_{1},x,y)=1152.
		\end{align*}
	Therefore, we have
	\begin{align}\label{eq:2.7}
	\bigl|H(3,1)(f^{-1})\bigr|\le 1.
	\end{align}
To establish sharpness, we show that the Koebe function is extremal in this case. Consider the Koebe function
	\[
	k(z)=\frac{z}{(1-z)^{2}}=z+2z^{2}+3z^{3}+4z^{4}+\cdots .
	\]
	Its inverse has the expansion
	\[
	k^{-1}(w)=w-2w^{2}+5w^{3}-14w^{4}+42w^{5}+\cdots .
	\]
	Hence
	\[
	A_{2}=-2,\qquad A_{3}=5,\qquad A_{4}=-14,\qquad A_{5}=42,
	\]
	and therefore
	\[
	H(3,1)(k^{-1})
	=
	2A_{2}A_{3}A_{4}-A_{3}^{3}-A_{4}^{2}+A_{3}A_{5}-A_{2}^{2}A_{5}
	=1.
	\]
	Thus the equality in \eqref{eq:2.7} is attained for the Koebe function, and hence the bound attained is sharp. 
	\end{proof}
	\bigskip
	
	\noindent\textbf{Compliance of Ethical Standards:}\\[2mm]
	\noindent\textbf{Conflict of interest.} The authors declare that there is no conflict of interest regarding the publication of this paper.\\[1mm]
	\noindent\textbf{Data availability statement.} Data sharing is not applicable to this article as no datasets were generated or analyzed during the current study.\\[1mm]
	\noindent\textbf{Authors contributions.} Both the authors have made equal contributions in reading, writing, and preparing the manuscript.\\[1mm]
	\noindent\textbf{Acknowledgment:}
	The second named author acknowledges financial support from the Council of Scientific and Industrial Research (CSIR), Government of India, through a CSIR Fellowship.

\end{document}